\documentclass[12pt]{amsart}
\usepackage{geometry, accents, mathtools, enumerate} 
\usepackage{ euscript, amssymb}

	\usepackage{xcolor}
	\usepackage{soul}
	\definecolor{lightred}{rgb}{1, .60,.60}
	
\renewcommand{\phi}{\varphi}

\renewcommand{\iff}{\leftrightarrow}

\newcommand{\SB}[1]{\mathbf{{\Sigma}}^{0}_{#1}}
\newcommand{\DB}[1]{\mathbf{{\Delta}}^{0}_{#1}}
\newcommand{\PB}[1]{\mathbf{{\Pi}}^{0}_{#1}}

\newcommand{\DL}[1]{\Delta^{0}_{#1}}

\newcommand{\seq}[1]{\langle{#1}\rangle}

\newcommand{\ck}{\omega^{ck}_1}

\newcommand{\N}{\mathbb{N}} 

\newcommand{\concat}{%
  \mathbin{\mathpalette\dotieconcat\relax}%
}
\newcommand{\dotieconcat}[2]{
  \text{\raisebox{.8ex}{$\smallfrown$}}%
}

\newcommand{\Pset}{\EuScript{P}}

\newcommand{\until}[1]{_{}\upharpoonright_{#1}}

\renewcommand{\subset}{\subseteq}

\usepackage{cancel,url}

\renewcommand{\until}[1]{\upharpoonright_{#1}}

\theoremstyle{plain}
\newtheorem{theorem}{Theorem}

\newcommand{\dom}{\mathrm{dom}}

\newtheorem{lemma}[theorem]{Lemma}

\newtheorem{proposition}[theorem]{Proposition}

\theoremstyle{definition}
\newtheorem{definition}[theorem]{Definition}

\newcommand{\emptystring}{\emptyset}

\newcommand{\FF}{\mathrm{F}}
\newcommand{\II}{\mathrm{I}}
\newcommand{\rank}{\mathrm{rank}}
\renewcommand{\P}{\mathbb{P}}
\newcommand{\symdiff}{\Delta}

\newcommand{\define}[1]{\emph{#1}}
\newcommand{\from}{\colon}
\DeclareMathOperator{\ran}{ran}
\newcommand{\powset}{\mathcal{P}}
\newcommand{\bfDelta}{\mathbf{\Delta}}
\newcommand{\bfSigma}{\mathbf{\Sigma}}
\newcommand{\bfPi}{\mathbf{\Pi}}
\newcommand{\cantor}{2^\omega}
\newcommand{\baire}{\omega^\omega}
\newcommand{\inters}{\cap}

\newcommand{\coord}[1]{({#1})} 

\newcommand{\lv}{*} 

\renewcommand{\supset}{\supseteq}

\theoremstyle{remark}

\usepackage{tikz, mathtools}

\usepackage{tkz-graph} 
\usetikzlibrary{arrows}

\SetVertexNormal[Shape = circle, LineWidth = 2pt] \SetUpEdge[lw = 1.5pt, color = black, labelcolor = white, labeltext = red, labelstyle = {sloped,draw,text=blue}]

\title{Jump Operations for Borel Graphs}

\thanks{The authors were supported by the Marsden fund of New Zealand. The
authors would also like to thank the Institute for Mathematical Sciences
and the National University of Singapore for their support at the 2015
summer school in logic, where several results of the paper were conceived.
The second author was partially supported by NSF grant DMS-1500974 and the
John Templeton Foundation under Award No.\ 15619.} 

\author{Adam R.~Day}
\address[Adam R.~Day]{Victoria University of Wellington}
\email{adam.day@vuw.ac.nz}

\author{Andrew S.~Marks}
\address[Andrew S.~Marks]{University of California, Los Angeles}
\email{marks@math.ucla.edu}

\date{\today}

\begin{document}

\begin{abstract}
 We investigate the class of bipartite Borel graphs organized by the
order of Borel homomorphism. We show that this class is unbounded by
finding a jump operator for Borel graphs analogous to a jump operator
of Louveau for Borel equivalence relations. The proof relies on a
non-separation result for iterated Fr\'echet ideals and filters due to Debs
and Saint Raymond. We give a new proof of this fact using 
effective descriptive set theory. We also investigate
an analogue of the Friedman-Stanley jump for Borel graphs. This
analogue does not yield a jump operator for bipartite Borel graphs. 
However, we use it to answer a question of Kechris and Marks
by showing that there is a Borel graph with no Borel homomorphism to a
locally countable Borel graph, but each of whose connected components
has a countable Borel coloring.
\end{abstract}

\maketitle

\section{Introduction}

The study of Borel graphs and their
combinatorial properties is a growing area of research
which has developed at the interface of descriptive set theory and
combinatorics, and has connections with probability, ergodic theory, and
the study of graph limits. The first systematic study of Borel graphs was
the paper of Kechris, Solecki and Todorcevic~\cite{MR1667145}, and~\cite{KM} is a recent survey of the field. 
We largely follow the conventions and notation from~\cite{KM}. A
\define{Borel graph} $G$ on $X$ is a graph whose vertex set $X$ is a Polish
space, and whose edge relation $G$ is Borel. We will typically abuse
notation and identify a graph with its edge relation.

Suppose $G$ and $H$ are graphs on the vertex sets $X$ and $Y$. Then a
\define{homomorphism} from 
$G$ to $H$ is a function $h \from X \to
Y$ such that $\forall x, y \in X (x G y \implies h(x) H h(y))$. That is, $f$ maps adjacent
vertices in $G$ to adjacent vertices in $H$. 
In classical
combinatorics, the book of Hell and Ne{\v{s}}et{\v{r}}il \cite{MR2089014}
is a good introduction to the theory of graph homomorphisms.

If $G$ and $H$ are Borel graphs on $X$ and $Y$, then we write $G
\preceq_B H$ if there is a Borel homomorphism $h \from X \to Y$ from $G$ to
$H$. Borel homomorphisms play a key role in the study of Borel colorings of Borel
graphs.
Recall that a \define{Borel coloring} of a Borel graph $G$ on a Polish
space $X$ is a
Borel function $c \from X \to Y$ to a Polish space $Y$ such that $\forall
x, y \in X (x G y \implies c(x) \neq c(y))$. So for example, if we take a
complete graph on $n$ vertices $K_n$, then $G \preceq_B K_n$
if and only if $G$ has a Borel $n$-coloring. 
More generally, if there is a Borel homomorphism from $G$ to $H$, then 
composing this homomorphism with a Borel coloring of
$H$ yields a Borel coloring of $G$. Thus, the problem of definably coloring
$G$ is no more difficult than coloring $H$. 
In an early
breakthrough result, Kechris, Solecki and Todorcevic \cite{MR1667145}
isolated a canonical graph $G_0$ having no Borel $\omega$-coloring such
that for every Borel (indeed, analytic) graph $G$, either $G$ has a
Borel coloring with countably many colors, or $G_0 \preceq_B G$.
Hence, $G_0$ is the canonical obstruction to an analytic graph having
countable Borel chromatic number. 

In this paper, we study the class of Borel graphs organized under
$\preceq_B$. We can view $\preceq_B$ as a general way of organizing all
Borel graphs by the relative difficulty of definably solving problems
on them, such as coloring, where solutions can be pulled back under homomorphisms. 

An important subclass of the Borel graphs is the bipartite graphs, or
equivalently, the graphs with no odd cycles. 
Such graphs are of particular interest in descriptive set theory since they
isolate the descriptive-set-theoretic difficulties that occur in definably coloring
Borel graphs, from the difficulties that arise if the graph has high
classical chromatic number (a graph has classical chromatic number $> 2$ if
and only if it is not bipartite). 
For example, the graph $G_0$ is acyclic and hence it is classically 2-colorable. 

Note that there may be no Borel witness to the
bipartiteness of a bipartite Borel graph. Indeed, Borel graphs $G$ that do possess
such a witness are trivial in the since that both $G \preceq_B K_2$ and
$K_2 \preceq_B G$ where $K_2$ is the complete graph on two vertices. 

One of our main theorems is that the class of
bipartite Borel graphs is unbounded under $\preceq_B$. In contrast, note that among all
Borel graphs, there is a trivial 
 example of a maximal Borel
graph under $\preceq_B$, specifically, the complete graph on any uncountable
Polish space.

Our proof of this
theorem is inspired by a jump operator of Louveau which comes from the theory
of Borel reducibility between Borel equivalence relations \cite{MR1327979}.
If $G$ is a Borel graph on a Polish space $X$, 
define $G^{\lv}$ to be the Borel graph defined on
$X^{\omega}$ where 
\[x G^{\lv} y \iff \exists n \forall i \geq n
(x\coord{i} G y\coord{i}).\] 
We show this induces a jump operation on bipartite
Borel graphs organized under $\preceq_B$. Note that if $G$
is bipartite, then so is $G^{\lv}$. 

\begin{theorem}\label{Louveau_jump}
  Let $G$ be a nontrivial bipartite analytic graph. Then $G \preceq_B
  G^{\lv}$ and $G^{\lv} \npreceq_B G$. 
\end{theorem}

By nontrivial, we mean that $G$ has at least one edge. Our theorem here in
fact generalizes to $G$ which are triangle-free.

Our proof of Theorem~\ref{Louveau_jump} requires an analysis of iterated
Fr\'echet filters that differs from that in \cite{MR1327979}. In
particular, our proof uses the following theorem due to Debs and Saint
Raymond.

\begin{theorem}[{\cite[Theorem 3.2, proof of Theorem 6.5]{MR2520152}}]\label{thm: non-separation}
  Let $F_\alpha, I_\alpha \subset \powset(\omega)$ be the $\alpha$\textsuperscript{th}
  iterate of the Fr\'echet ideal and Fr\'echet filter. Then 
  $I_\alpha$ and $F_\alpha$ cannot be separated by $\bfPi^0_{\alpha+1}$
  sets (but they can be separated by disjoint $\bfSigma^0_{\alpha+1}$
  sets).
\end{theorem}

In Section~\ref{Fr_section}, we give a new proof of this theorem which uses
different methods than \cite{MR2520152}. In particular, it uses
effective descriptive set theory and an analysis of suitably chosen generic
reals.

A more well-known jump operation on Borel equivalence relations is the 
Friedman-Stanley jump~\cite{MR1011177}. We can also define an operation on
Borel graphs similar to the Friedman-Stanley jump. 
Suppose $G$ is a Borel graph on a Polish space $X$. Define $X^{+} \subset
X^{\omega}$ by $x \in X^+$ if $i_1 \neq i_2$ implies $x\coord{i_1}$ and
$x\coord{i_2}$ are in different connected components of $G$. 
Now we define the Borel graph $G^+$ on $X^+$ by:
\[x G^+ y \iff \forall i \exists j (x\coord{i} G
y\coord{j}) \land \forall i
\exists j (y\coord{i} G x\coord{j}).\]
Now it is easy to see that if $G$ is a Borel
graph with countably many connected components, then there is a Borel
homomorphism from $G^+$ to $G$; take a Borel ordering $<$ on the
connected components of $G$ of ordertype $\omega$, and then define a
homomorphism $h$ from $G^+$ to $G$ by letting $h(x)$ be the $<$-least
$x\coord{j}$ in the sequence $(x\coord{i})_{i \in \omega}$.
Similarly, for such a $G$, there is a Borel
homomorphism from $G^{++}$ (which has continuum many connected components)
to $G^{+}$.  Characterizing the Borel
graphs $G$ such that $G^+ \npreceq_B G$ remains an open question.
However, we show that in at least one special case $G^+ \npreceq_B G$:
when $G$ is a nontrivial locally countable bipartite Borel graph with
meager connectedness relation. More generally,
\begin{proposition}\label{FS_no_homomorphism}
  Suppose $G$ is a Borel graph on a Polish space $X$ whose connectedness
  relation is meager and there is a Borel homeomorphism $T \from X \to X$ such that 
  $T(x) \mathrel{G} x$ for every $x \in X$. Then for every locally
  countable Borel graph $H$, $G^+ \npreceq_B H$. 
\end{proposition}

We use this proposition to prove the following theorem, which 
answers a question of Kechris and Marks~\cite[Section 3.(H)]{KM}:

\begin{theorem}\label{KM_answer}
  There is a Borel graph $G$ such that for any connected component $C$ of
  $G$, $G \restriction C$ has a Borel $\omega$-coloring. 
  However, there is no Borel graph homomorphism from
  $G$ to a locally countable Borel graph. 
\end{theorem}

The paper is organized as follows. In Section~\ref{FS_section}, we use a
Baire category argument to prove a lemma about the Friedman-Stanley jump
for Borel graphs, and then use this to prove Theorem~\ref{KM_answer}. In
Section~\ref{L_section}, we prove 
Theorem~\ref{Louveau_jump} about the Louveau jump, modulo Theorem~\ref{thm: non-separation}. Finally, in
Section~\ref{Fr_section}, we prove Theorem~\ref{thm: non-separation} concerning how
hard it is to separate the iterated Fr\'echet ideal and filter, using
effective descriptive set theory. 

\section{A Friedman-Stanley Jump for Borel Graphs}\label{FS_section}

We begin by considering the jump $G \mapsto G^+$ for Borel graphs defined
above. First, we show that Proposition~\ref{FS_no_homomorphism} implies
Theorem~\ref{KM_answer}.

\begin{proof}[Proof of Theorem~\ref{KM_answer}:]
Let $T \from X \to X$ be a fixed-point free Borel homeomorphism of a
perfect Polish space $X$, which induces the Borel graph $G_T$ on $X$ where
$x \mathrel{G_T} y$ if $T(x) = y$ or $T(y) = x$. Let $G =
(G_T)^+$. Now if $C$ is a connected component of $G$, then fixing
some $y \in C$, map each $x \in C$ to the unique $x(i)$ such that $x(i)$ is
in the same $G$-component as $y(0)$. This is a countable Borel coloring, since
connected components of $G$ are countable. However, there is no Borel
homomorphism from $G$ to a locally countable Borel graph $H$ by
Proposition~\ref{FS_no_homomorphism}. 
\end{proof}

We now use a Baire category argument to prove
Proposition~\ref{FS_no_homomorphism}.

\begin{proof}[Proof of Proposition~\ref{FS_no_homomorphism}:]
A basis for $X^\omega$ consists of the basic open sets $N_{U_0, \ldots,
U_{k-1}} = \{x \in X^\omega : \forall i < k (x\coord{i} \in U_i)\}$, where
$U_0, \ldots, U_{k-1} \subset X$ are open. Here we say $N_{U_0, \ldots,
U_{k-1}}$ \define{restricts coordinates less than $k$}. Let $H$ be a
locally countable Borel graph on a Polish space $Y$. By changing the
topology on $Y$, we may assume that $H$ is generated by countably many
Borel homeomorphisms $S_0, S_1, \ldots$ (see {\cite[Exercise
13.5]{MR1321597}}).

Suppose that $p \from \omega \to \omega$ is bijection. Then define the
function $T_p \from X^\omega \to X^\omega$ by 
\[T_p(x)\coord{n} = T(x\coord{p(n)})\]
for all $n \in \omega$. Note that since $T$ is a homeomorphism, $T_p$ is a
homeomorphism which sends basic
open sets to basic open sets, and $T_p(x) \mathrel{G^+} x$ for every $x \in X^+$.

Suppose $h$ is a Borel function from $X^+$ to $Y$. We will use a Baire
category argument to show that $h$ is not a homomorphism from $G^+$ to $H$. Note that $X^+$
is a comeager subset of $X^\omega$, since the connectedness relation of $G$
is meager. Fix a comeager set $C \subset X^+$ such that $h
\restriction C$ is continuous. Our proof breaks down into three cases.

\smallskip \noindent \textit{Case 1:}
Assume that $h \restriction C$ is constant on some open set $U$. 
We will
find $x,y \in U$ such that $x \mathrel{G^+} y$, but $h(x) = h(y)$.
We
may assume $U$ is basic open and 
restricts coordinates less than $k$. Let $p \from \omega \to \omega$ be
the permutation which swaps the interval $[0,k)$ with the interval
$[k,2k)$:
\[p(n) = \begin{cases}
n+k &\mbox{if $n< k$}\\
n-k &\mbox{if $k \le n <2k$}\\
n &\mbox{if $n >2k$}.
\end{cases}\]

Then $T_p^{-1}(U)$ is a basic open set which restricts
coordinates in $[k,2k)$, and so $V = T_p^{-1}(U) \inters U
\neq \emptyset$. We claim that comeagerly many $x \in V$ have $h(x) =
h(T_p(x))$. This is because $h \restriction C$ is constant on $U$, and
since $C \inters U$ is comeager in $U$ and $T_p$ is a homeomorphism, 
$T_p^{-1}(C \inters U)$ is comeager in $T_p^{-1}(U)$. 
Hence comeagerly many $x \in V$ have $x \in C$ and $T_p(x) \in C$, and so $h(x) = h(T_p(x))$.

\smallskip \noindent\textit{Case 2:} There is a basic open set $U = N_{U_0, \ldots,
U_{k-1}}$, such that the value of $h \restriction C$ on $x \in N_{U_0, \ldots, U_{k-1}}$ depends only on the first $k$ coordinates of $x$. That is, if we
take any two basic open neighborhoods
$N_{V_0, \ldots, V_n}, N_{V'_0, \ldots, V'_n} \subset N_{U_0, \ldots,
U_{k-1}}$
where $V_i = V'_i$ for $i < k$, then $h(C \inters N_{V_0, \ldots, V_n})
\inters h(C \inters N_{V'_0, \ldots, V'_n}) \neq \emptyset$. 
In this case, we will find  $x, y \in U$ such that $y = T_p(x)$ and so $x
\mathrel{G^+} y$, but $h(x) \mathrel{\cancel{G}} h(y)$. 

Let $p$ be the permutation above which swaps the interval $[0,k)$ with
$[k,2k)$. 
Let $V = U \inters T_p^{-1}(U)$ which is a nonempty open set. Let $A_m =
\{x \in X^+ :
S_m(h(x)) \neq h(T_p(x))\}$. It suffices to show that each $A_m$ is comeager in
$V$, since then the set of $x$ such that
$h(x) \mathrel{\cancel{G}} h(T_p(x))$ is comeager in $V$. Fix $m \in
\omega$. Given an arbitrary open subset $W \subset V$, it suffices to
show that $A_m$ is
nonmeager in $W$. Since $h \restriction C$ is continuous and not constant on any open
set and $S_m$ is a homeomorphism, we can find open sets $W_0, W_1 \subset W$ such
that $S_m(h(C \inters W_0))
\inters h(C \inters T_p(W_1)) = \emptyset$. By refining, we may
assume that $W_0$ and $W_1$ are basic open sets. 
But since $h \restriction C$ depends only on the
first $k$ coordinates, we may assume $W_0$ and $T_p(W_1)$
restrict only coordinates
less than $k$, and so $W_1$ only restricts coordinates in $[k,2k)$. Thus,
$W_0 \inters W_1$ is open and nonempty, and 
comeagerly many $x$ in $W_0 \inters W_1$ are in $A_m$. 

\smallskip \noindent \textit{Case 3:} Assume Case 1 and Case 2 do not
occur. In this case, instead of specifying a particular permutation $p\colon
\omega \rightarrow \omega$ and considering $h$ applied to a generic $x$ and
$T_p(x)$, we will consider a generic permutation $p$.
More
precisely, let $Y \subset \baire$ be the closed set of injective functions
from $\omega \to \omega$. A basis for $Y$ consists of the sets $N_q = \{p
\in Y : p \supset q\}$ for some finite partial injection $q$. Now the set
of $p \in Y$ such that $p$ is a bijection from $\omega$ to $\omega$ is
dense $G_\delta$. To complete the proof, we will show that  comeagerly many
$(x,p) \in X^+ \times Y$ have the property that $h(x) \mathrel{\cancel{G}}
h(T_p(x))$.

For each $n$, let $B_n = \{(x,p) \in X^+ \times Y : S_n(h(x)) \neq
h(T_p(x))\}$. It suffices to show that $B_n$ is nonmeager in every open set $U \times
V \subset X^+ \times Y$. By refining, we may assume there is some $k$ so that $U$ is a basic open set which
restricts coordinates less than $k$, and $V$ is a basic open set of the form $V
= N_q$ where $q$ is a 
finite partial injection with $\dom(q) = \ran(q) = k$.

Because $U$ only restricts coordinates less than $k$, and $\dom(q) =
\ran(q) = k$ we have that for all $p , p' \in N_q$, $T_p(U) = T_{p'}(U)$.
Hence we can define $T_q(U)$ to be this common value (though $q$
is only a finite partial function). This can be done in any situation where the range and domain of $q$ include all coordinates restricted by a basic open set $U$.

Let $W=T_q(U)$. 
We claim that there are basic open sets $U' \subset U$ and $W'\subset W$,
$q' \supset q$ and $k'$ such 
that:
\begin{enumerate}
\item  $\dom(q')=\ran(q') \ge k'$.
\item $U'$ and $W'$ only restrict coordinates less than $k'$.
\item $T_{q'}(U') \cap W' \ne \emptyset$.
\item $S_n(h(U' \cap C)) \cap h(W' \cap C) = \emptyset$.
\end{enumerate}

We establish this claim as follows. Since we are not in Case 1 or Case 2, we can
find basic open $U_1 = N_{V_0, \ldots, V_{j}}$ and $U_2 = N_{V'_0, \ldots,
V'_{j}}$ with $U_1, U_2 \subset U$ such that $V_i = V'_i$ for all $i <k$, and $h(U_1 \inters
C) \inters h(U_2 \inters C) = \emptyset$. Hence, $S_n(h(U_1 \inters C))
\inters S_n(h(U_2 \inters C)) = \emptyset$ since $S_n$ is a homeomorphism.
Let $U^* = N_{V_0, \ldots, V_{k-1}}$.  We can find a basic open $W' \subset T_q(U^*)$   such that $h(W' \cap C)$ is disjoint from $S_n(h(U_i \cap C))$ for some $i \in \{0,1\}$. Let $U'=U_i$ and so part (iv) of the claim holds.  As  $U^* \subset U$, this definition of $W'$ ensures that $W' \subset W$.
Let $k'$ be such that $W'$ and $U'$ do not restrict coordinates less than $k'$. 

Let $c = k'-k$. Define $q'$ as follows
\[q'(i) = \begin{cases} q(i) &\mbox{if } i <k\\
i+c &\mbox{if } k \le i <k+c \\
i-c &\mbox{if } k+c \le i <k+2c.
\end{cases}\]
It remains to show part (iii) of the claim. If  $T_{q'}(U') \cap W' = \emptyset$ then these sets must place incompatible requirements on some coordinate. By construction of $q'$, this coordinate must be less than $k$ (as $W'$ restricts coordinates in $[0, \ldots, k +c)$ and $T_{q'}(U')$ restricts coordinates in 
$[0, \ldots, k) \cup [k+c, \ldots, k+2c)$). However this is not the case as $W' \subset T_q(U^*)$ and $T_q(U^*)$ places the same restrictions on coordinates less than $k$ as does $T_q(U')$.

Now using the claim we will complete the proof. Let $r$ be the inverse of $q'$. Take any $p \in N_{q'}$. 
Take any $x \in  U' \cap T_{r}(W')$. If $x$ is also in the comeager set $  C \cap T_p^{-1}(C)$, then $T_p(x) \in W' \cap C$ and $x \in U' \cap C$. Hence $S_n(h(x)) \ne h(T_p(x))$. Thus for comeagerly many $x \in U' \cap T_{r}(W')$ we have that $(x, p) \in B_n$. Thus $B_n \cap ((U' \cap T_{r}(W')) \times N_{q'})$ is comeager and so $B_n$ is nonmeager in $U \times V$.

\end{proof}

\section{The Louveau jump for Borel graphs} 
\label{L_section}

In this section, we consider the Louveau jump $G \mapsto G^\lv$ defined in
the introduction, and we prove Theorem~\ref{Louveau_jump}, that if $G$ is a
bipartite Borel graph, there is no
Borel homomorphism from $G^\lv$ to $G$. Note that among all graphs, there is
not always a homomorphism from $G^\lv$ to $G$ (consider the $G$ that is the
disjoint union of the complete graphs $K_n$ for every $n$). However, if $G$
is bipartite, then using the axiom of choice we can always find a
homomorphism from $G^\lv$ to $G$. This is because if $G^\lv$ has a cycle of
length $n$, then $G$ clearly has a cycle of length $n$. Hence, if $G$ is
bipartite, then $G^\lv$ is also bipartite, so using choice we can find a
homomorphism of $G^\lv$ into any two vertices of $G$ connected by an edge. 

Like Louveau's argument in \cite{MR1327979}, our proof will involve
potential complexity in a key way. Recall that if $\Gamma$ is a class of
subsets of Polish spaces,
$X$ is a Polish space with topology $\tau$, and $A \subset X
\times X$, then $A$ is said to be \define{potentially $\Gamma$} if there is
a Polish topology $\tau'$ on $X$ inducing the same Borel $\sigma$-algebra
as $\tau$ such that $A$ is $\Gamma$ as a subset of $X \times X$ with topology
$\tau' \times \tau'$. 
Louveau's proof relies on the fact that if $E$ and $F$ are equivalence
relations on Polish spaces $X$ and $Y$ and $\Gamma$ is a class of sets
closed under continuous preimages, then if $E \leq_B F$ and $F$ is
potentially $\Gamma$, then so is $E$. However, the analogous fact fails
for Borel graphs under Borel homomorphism: there are Borel graphs of
arbitrarily high potential complexity which admit Borel homomorphisms to
Borel graphs of low potential complexity (e.g. to $K_2$). So instead, we will
consider the complexity of separating pairs of points of distance one and
distance two.

\begin{definition}
  If $G$ is a graph on $X$, let $D_{G,n}$ be the set of pairs $(x,y) \in X
  \times X$ such that there is a path from $x$ to $y$ of length $n$ in $G$. 
\end{definition}

For example, if $G$ is an analytic triangle-free graph on a Polish space
$X$, then $D_{G,1}$ and $D_{G,2}$ are disjoint analytic sets which can
therefore be separated by a Borel set. 

\begin{proposition}\label{potential_complexity}
  Suppose $G$ and $H$ are triangle-free graphs on the
  Polish spaces $X$ and $Y$, and there is a Borel homomorphism from $G$ to
  $H$. Then if $D_{H,1}$ and $D_{H,2}$ can be separated by a potentially
  $\bfSigma^0_\alpha$ (resp. $\bfPi^0_\alpha)$ set, then $D_{G,1}$ and
  $D_{G,2}$ can also be separated by a potentially $\bfSigma^0_\alpha$
  (resp. $\bfPi^0_\alpha$) set.
\end{proposition}
\begin{proof}
  Let $h$ be a Borel homomorphism from $G$ to $H$. By changing topology,
  we may assume that $h$ is continuous (see \cite[Theorem 3.11]{MR1321597}).
  Since $h$ is a homomorphism, if there is a path from $x$ to $y$ in $G$ of
  length $n$, then there is a path from $h(x)$ to $h(y)$ in $H$ of length
  $n$. Hence, the preimage of a set separating $D_{H,1}$ and $D_{H,2}$
  under the function $(x,y) \mapsto (h(x),h(y))$ will separate $D_{G,1}$
  and $D_{G,2}$. The result follows since the classes $\bfSigma^0_\alpha$
  and $\bfPi^0_\alpha$ are closed under taking continuous preimages. 
\end{proof}

Next, we will define a transfinite way to iterate the operation $G
\mapsto G^\lv$. In what follows, for each countable limit ordinal $\lambda$,
let $\pi_\lambda:\omega \to \alpha$ be an increasing and cofinal
function. For each successor ordinal $\alpha+1$, let
$\pi_{\alpha+1}:\omega \to \alpha +1$ be the constantly $\alpha$
function. These functions will allow us to simultaneously handle limit and
successor cases in a uniform way.

Suppose $(G_i)_{i \in \omega}$ is a countable sequence of graphs. Define
$(G_i)^\lv$ to be the graph whose vertices are elements of $\prod_i G_i$
and such that  $x (G_i)^\lv y$ if there exists an $n$ such that for all $m
> n$ we have $x\coord{m} G_m y\coord{m}$. Now we define $G^\alpha$ for
countable ordinals $\alpha$ as follows.
\begin{align*}
&G^0 = G \\
&G^{\alpha} = (G^{\pi_\alpha(n)})^* \quad \mbox{if }\alpha>0
\end{align*}
Note that under this definition, $G^1=G^\lv$.

\begin{lemma} \label{lem: alpha homomorphism}
 Let $G$ and $H$ be graphs. If there is a Borel homomorphism from $G$ to
 $H$, then for all countable ordinals $\alpha$, there is a Borel
 homomorphism from $G^\alpha$ to~$H^\alpha$. \qed
\end{lemma}

\begin{lemma}
Let $(G_i)_{i \in \omega}$ be a countable sequence of graphs. Assume for
each $i \in \omega$ that $h_i:G_i \to G$ is a Borel homomorphism.
Then there is a Borel homomorphism from $(G_i)^\lv$ to $G^\lv$. 
\end{lemma}
\begin{proof}
Let $X$ be the vertex set of $G$ and $X_i$ be the vertex set of $G_i$. 
Then define the Borel homomorphism
$g: \prod_i X_i \to \prod_i X$ from $(G_i)^\lv$ to $G^\lv$ by 
 $g(x)\coord{n}= h_n(x\coord{n})$ i.e.\ apply to the $n$\textsuperscript{th} vertex in the sequence the mapping $h_n$.
Now assume that $x \mathrel{(G_i)^\lv} y$. Then for some $n$ for all $m >n$ we have
that $x\coord{m} \mathrel{G_m} y \coord{m}$. Hence 
$h_m(x\coord{m}) \mathrel{G} h_m(y\coord{m})$ and so $h(x) \mathrel{G^\lv} h(y)$.
\end{proof}

\begin{lemma}
\label{lem: alpha to G}
If there is a Borel homomorphism from $G^\lv$ into $G$, then for all countable ordinals $\alpha$, there is a Borel homomorphism from $G^\alpha$ into $G$.
\end{lemma}
\begin{proof}
If for all $\beta < \alpha$ there is a Borel homomorphism from $G^\beta$
to $G$, then by the previous lemma, there is a Borel homomorphism from
$G^\alpha$ to $G^\lv$. Now this homomorphism can be composed with the
homomorphism from $G^\lv$ to $G$ to obtain a Borel homomorphism from $G^\alpha$ to $G$.
\end{proof}

As usual, let us denote by $K_2$ the complete graph on two vertices. The graphs $K^\alpha_2$ for countable ordinals $\alpha$ will play a key part in the main theorem of this section. We will investigate these graphs further. First we will define the spaces on which these graphs live.  Inductively define $X^\alpha$ as follows.
\begin{align*}
&X^0 = 2 \\
& X^{\alpha} = \prod_n X^{\pi_\alpha(n)} \quad\mbox{if }\alpha >0
\end{align*}
It should be clear from the definition of $G^\alpha$ that $K^\alpha_2$ is a graph on $X^\alpha$. If $\alpha > 0$, then the space $X^\alpha$ is clearly homeomorphic to~$2^\omega$. We want to define a homeomorphism from $X^\alpha$ to $2^\omega$ that respects the manner in which $X^\alpha$ has been inductively constructed.  
We define $\gamma^\alpha:X^\alpha \to 2^\omega$ inductively as follows. 
\begin{enumerate}
\item  $\gamma^1$ is the identity map.
\item  For $\alpha >1$,  we define  $\gamma^{\alpha }(x)$ to be the
sequence $z$ such that for all $m$ and $n$, $z\coord{\seq{m,n}} =
(\gamma^{\pi_\alpha(m)}(x\coord{m}))(n)$.
\end{enumerate}

The edge relation on $K^\alpha_2$ can be though of in terms of the iterated
Fr\'{e}chet filter. The $\alpha$ iterate of the Fr\'{e}chet filter is a subset of $X^\alpha$. It  will be denoted by $\FF^\alpha$ and it is defined inductively as follows.  
 \begin{align*}
&\FF^0 = \{1\} \\
&\FF^{\alpha} = \{x \in X^{\alpha} \colon \exists n\, \forall m >n (x
\coord{m} \in \FF^{\pi_\alpha(m)})\}\quad \mbox{if $\alpha >0$.}
\end{align*}
Note that $\FF^1$ is the standard Fr\'{e}chet, or cofinite, filter. Similarly we can define the iterated Fr\'{e}chet ideals. 
 \begin{align*}
&\II^0 = \{0\} \\
&\II^{\alpha} = \{x \in X^{\alpha} \colon \exists n\, \forall m >n (x(m) \in \II^{\pi_\alpha(m)})\}\quad \mbox{if $\alpha >0$.}
\end{align*}

If $x, y$ are elements of $2^\omega$ then we denote by $x \symdiff y$ the
binary sequence $z$ where $z(n) = 0$ if $x(n)=y(n)$ and $z(n) =1$
otherwise. The notation $\symdiff$ is used because  $x \symdiff y$ is the
symmetric difference of $x$ and $y$ when $x$ and $y$ are considered as
subsets of $\omega$ under the standard bijection between $2^\omega$
and~$\Pset(\omega)$. 

We can use the fact that $X^\alpha$ is homeomorphic to $2^\omega$ in order
to induce a symmetric difference operation on $X^\alpha$ when $\alpha >0$.
Define the operation $\symdiff^\alpha$ on $X^\alpha$ by $x\symdiff^\alpha
y= z$ if and only if $\gamma^\alpha(x)\symdiff \gamma^\alpha(y) =
\gamma^\alpha(z)$. For the case $\alpha =0$, for $x,y \in X^0$, define $x
\symdiff^0 y = 0$ if $x=y$ and $x \symdiff y = 1$ if $x\ne y$.

Since the homeomorphism $\gamma^\alpha$ is defined simply by permuting
indices from a
product of infinitely many copies of $2$ to yield $\cantor$, we have that
symmetric difference operation commutes with $\gamma^\alpha$. 

We will make use of the following properties of $\symdiff^\alpha$. 
\begin{enumerate}
 \item The operation $\symdiff^\alpha$ is  associative and commutative.
 \item $x\symdiff^\alpha x \symdiff^\alpha y = y$.
\item For all  $\alpha>0$,   $(x \symdiff^{\alpha} y)(n) = x(n) \symdiff^{\pi_\alpha(n)} y(n)$.   
 \item For all $\alpha$, if $x \in \II^\alpha$, then $x \symdiff^\alpha y
 \in \FF^\alpha$
 if and only if $y \in \FF^\alpha$.
\end{enumerate}

The following lemma explains how to define the edge relation in $K^\alpha_2$ using $\FF^\alpha$. 

\begin{lemma}
 For all $\alpha$ and $x, y \in X^\alpha$, 
 $x K^\alpha_2 y$ if and only if $x\symdiff^\alpha y \in \FF^\alpha$.
 \end{lemma}
\begin{proof}
 This result holds trivially for the case $\alpha =0$. Take $\alpha >0$ and assume the result holds for all smaller ordinals. Consider $x, y \in X^{\alpha}$. We have that 
\begin{align*}
 x K^{\alpha}_2 y & \Leftrightarrow \exists n\, \forall m >n (x(m) K^{\pi_\alpha(m)}_2 y(m))\\
 & \Leftrightarrow \exists n \,\forall m >n (x(m)\symdiff^{\pi_\alpha(m)} y(m)) \in \FF^{\pi_\alpha(m)}\\
 & \Leftrightarrow \exists n \,\forall m >n (x \symdiff^\alpha y)(m) \in \FF^{\pi_\alpha(m)}\\
& \Leftrightarrow  x \symdiff^\alpha y \in  \FF^{\alpha}.\qedhere
\end{align*}
\end{proof}

Similarly, we have the following:

\begin{lemma}
\label{lem path length 2}
 For all $\alpha$ and $x, y \in X^\alpha$, 
 $x\symdiff^\alpha y \in \II^\alpha$ if and only if there is a path of length $2$ from $x$ to $y$ in $K^\alpha_2$.
 \end{lemma}
\begin{proof} Suppose $x K_2^\alpha z$. By the above lemma, this is true if and only if $x \symdiff^\alpha z \in
\FF^\alpha$. Letting $w = x \symdiff^\alpha z$, we have $z K_2^\alpha y$ if and
only if $z \symdiff^\alpha y \in \FF^\alpha$. Finally, note that $z
\symdiff^\alpha y = (w \symdiff^\alpha x) \symdiff^\alpha y = w
\symdiff^\alpha (x \symdiff^\alpha y)$ which is in $\FF^\alpha$ if and only
if $x \symdiff^\alpha y \in \II^\alpha$, since $w \in \FF^\alpha$.
\end{proof}

\begin{lemma}
\label{lem: generic pairold}
Let $C \subset \cantor$ be a comeager set. Then there is a continuous map $f:2^\omega \rightarrow 2^\omega \times 2^\omega$ such that 
for all~$x$:
\begin{enumerate}
\item $f_0(x) \in C$.
\item $f_1(x) \in C$.
\item $f_0(x) \symdiff f_1(x) =x$.
\end{enumerate}
\end{lemma}
\begin{proof} Let $(E_n)$ be a sequence of dense open sets such that
$\bigcap_n E_n \subset C$. We will use $p_n$ and $q_n$ to denote finite strings. 

Fix $x$. Let $p_0$ and $q_0$ be the empty string $\emptystring$. Now given $p_n$, take $p_{n+1}' < p_n$, such that $N_{p_{n+1}'}  \subset E_n$. Define $q_{n+1}' < q_n$ 
so that for all $i < |q_{n+1}'|$, $x(i) =1$ if and only if $p_{n+1}'(i) \ne q_{n+1}'(i)$.
Now take  $q_{n+1} < q_{n+1}'$, such that $N_{q_{n+1}}  \subset E_n$. Define $p_{n+1} < p_{n+1}'$
so that for all $i < |q_{n+1}|$, $x(i) =1$ if and only if $p_{n+1}(i) \ne q_{n+1}(i)$. 

By induction, for all $n$, we have that $p_{n+1}< p_n$, and $q_{n+1} < q_n$. Let $f_0(x) = \lim_n p_n$ and  $f_1(x) = \lim_n q_n$. Because we only used finitely many bits of $x$ at each stage of the construction, this construction gives a continuous map with the desired properties. 
\end{proof}

\begin{lemma}
\label{lem: generic pair}
For any $\alpha$ and any comeager set $C$ in $X^\alpha$, there is a continuous map $f:X^\alpha \rightarrow X^\alpha \times X^\alpha$ such that 
for all~$x$:
\begin{enumerate}
\item $f_0(x) \in C$.
\item $f_1(x) \in C$.
\item $f_0(x) \symdiff^\alpha f_1(x) =x$.
\end{enumerate}
\end{lemma}
\begin{proof}
This follows from the previous lemma and the fact that 
there is a homeomorphism $\gamma^\alpha:X^\alpha \rightarrow 2^\omega$ such that for $x,y \in X^\alpha$, $x\symdiff^\alpha y$ if and only if $\gamma^\alpha(x) \symdiff \gamma^\alpha(y)$. 
\end{proof}

We now have the following lemma:

\begin{lemma}
  Suppose $\alpha < \omega_1$. Then $D_{K_2^\alpha, 1}$ and $D_{K_2^\alpha,
  2}$ can be separated by a potentially $\bfSigma^0_\alpha$ set, but not by a
  potentially $\bfDelta^0_\alpha$ set.
\end{lemma}
\begin{proof}
  By the above lemmas, we have $D_{K_2^\alpha, 1} = \{(x,y) \colon x
  \symdiff^\alpha y \in \FF^\alpha\}$, and $D_{K_2^\alpha, 2} = \{(x,y)
  \colon x \symdiff^\alpha y \in \II^\alpha\}$. Now the function 
  $(x,y) \mapsto x \symdiff^\alpha y$ is continuous and so by
  Theorem~\ref{thm: non-separation}, 
  $D_{K_2^\alpha, 1}$ and $D_{K_2^\alpha, 2}$ can be separated by a
  potentially $\bfSigma^0_\alpha$ set.

  Suppose there was a topology $\tau'$ on $X^\alpha$ so that
  $D_{K_2^\alpha, 1}$ and $D_{K_2^\alpha, 2}$ were separable by a
  potentially $\bfDelta^0_\alpha$ set.
  Let $\tau$ be the usual topology on $X^\alpha$. 
  Then there is a comeager set $C$ on which the
  identity function $id \from (X^\alpha,\tau) \to (X^\alpha,\tau')$ is
  continuous. Let $f_0$ and $f_1$ be as in Lemma~\ref{lem: generic pair}.
  Then $id \circ f_0$ and $id \circ f_1$ are continuous, and taking the
  preimage of the separating set under $x \mapsto (id \circ f_0(x), id
  \circ f_1(x))$ would separate $\II^\alpha$ and $\FF^\alpha$ by a
  $\bfDelta^0_\alpha$ set, contradicting Theorem~\ref{thm: non-separation}. 
\end{proof}

We are now ready to prove (a slight generalization of)
Theorem~\ref{Louveau_jump} from the introduction.

\begin{theorem}
Let $G$ be a triangle-free analytic graph on a Polish space, with at least one
edge.  Then there is no Borel homomorphism from $G^\lv$ to $G$.
\end{theorem}
\begin{proof}
  Since $D_{G,1}$ and $D_{G,2}$ are analytic sets, they can be separated by
  a Borel set which is $\bfDelta^0_\beta$ for some countable ordinal $\beta$. Let
  $\alpha$ be such that $\FF^\alpha$ and $\II^\alpha$ are not separable by
  a $\bfDelta^0_\beta$ set. 
 
  Assume that there is a Borel homomorphism from $G^\lv$ into $G$. By
  Lemma~\ref{lem: alpha to G}, if there is a Borel homomorphism from
  $G^\lv$ into $G$, then there is a Borel homomorphism from $G^\alpha$ to
  $G$. By Lemma~\ref{lem: alpha homomorphism}, there is Borel homomorphism
  from $K^\alpha_2$ to $G^\alpha$ and hence a Borel homomorphism from
  $K^\alpha_2$ to $G$. By Proposition~\ref{potential_complexity}, this
  implies $D_{K^\alpha_2,1}$ and $D_{K^\alpha_2,2}$ can be separated by a
  Borel set which is potentially $\bfDelta^0_\beta$, contradicting our choice of
  $\alpha$.
\end{proof}

\section{Separating the Iterated Fr\'{e}chet Filters and Ideals}
\label{Fr_section}

In this section, we give a new proof of Theorem~\ref{thm: non-separation}. 

We begin by proving
that $\FF^\alpha$ and $\II^\alpha$ can be separated by disjoint 
$\SB{\alpha+1}$ sets. The idea here is based on a trick for switching
quantifier order which is encapsulated in the following lemma.

\begin{lemma}
\label{lem: separating core}
For each $n,k \in \N$ let $A_{n,k}$ and $B_{n,k}$ be disjoint $\SB{\alpha}$ subsets of a Polish space $X_{n,k}$. Let 
\begin{align*}
A &= \{x \in  \textstyle\prod X_{n, k}\colon  (\exists m)(\forall n >m)(\exists j)(\forall k >j)(
x(n, k) \in A_{n,k})\},\\
B &=\{x \in \textstyle\prod X_{n, k} \colon  (\exists m)(\forall n >m)(\exists j)(\forall k >j)(
x(n , k) \in B_{n,k})\}.
\end{align*}
Then there are $A^*$, $B^*$, disjoint $\SB{\alpha+2}$ subsets of $\textstyle\prod X_{n, k}$ such that $A \subset A^*$ and $B \subset B^*$.
\end{lemma}
\begin{proof}
 Let us define the sets
 \begin{align*}
A' &= \{x \in  \textstyle\prod X_{n, k}\colon  (\exists m)(\forall n >m)(\forall j)(\exists k >j)(
x(n, k) \in A_{n,k})\},\\
B' &=\{x \in \textstyle\prod X_{n, k} \colon  (\exists m)(\forall n >m)(\forall j)(\exists k >j)(
x(n , k) \in B_{n,k})\}.
\end{align*}

The sets $A'$ and $B'$ are $\SB{\alpha+2}$. Now $A \subset A'$ because if for some $n$ there are cofinitely many $k$ such that  $x(n, k) \in A_{n,k}$ then there are infinitely many $k$ such that $x(n, k) \in A_{n,k}$.
Also $B \cap A' = \emptyset$ because if there are  infinitely many $k$ such that $x(n, k) \in A_{n,k}$, then there cannot be cofinitely many $k$ such that $x(n, k) \in B_{n,k}$ (as $A_{n,k}$ and $B_{n,k}$ are disjoint).
Similarly we have that $B \subset B'$ and $A \cap B'=\emptyset$. Finally apply separation to obtain disjoint $\SB{\alpha+2}$ sets $A^*$ and $B^*$ with $A^* \subset A'$, 
$B^* \subset B'$, and $A^* \cup B^* = A' \cup B'$. Hence $A \subset A^*$ and $B\subset B^*$.
\end{proof}

\begin{lemma}
For all $\alpha$, 
$\FF^\alpha$ and $\II^\alpha$ can be separated by disjoint $\SB{\alpha+1}$ sets. \end{lemma}
\begin{proof} If $\alpha$ is $0$, $1$, or a limit ordinal then there  is nothing to prove because the sets $\FF^\alpha$ and $\II^\alpha$ already have the required complexity. 

If $\alpha = \beta+1$ where $\beta$ is a limit ordinal, then
\begin{align*}
\FF^\alpha &= \{x \in X^\alpha \colon (\exists m)(\forall n >m)(\exists j)(\forall k >j)(
x(n)(k) \in \FF^{\pi_\beta(k)})\},\\
\II^\alpha &= \{x \in X^\alpha \colon (\exists m)(\forall n >m)(\exists j)(\forall k >j)(
x(n)(k) \in \II^{\pi_\beta(k)})\}.
\end{align*}
As for all $k$, the sets $\FF^{\pi_\beta(k)}$ and $\II^{\pi_\beta(k)}$ are
disjoint $\SB{\beta}$, we can apply Lemma~\ref{lem: separating core}, to find disjoint $\SB{\alpha+1}$ sets that separate $\FF^\alpha$ and $\II^\alpha$.  

Finally, we need to consider the case that $\alpha$ is $\beta +n$ where $n
\ge 2$ and $\beta$ is a limit ordinal or $0$. In this situation we have that 
\begin{align*}
\FF^\alpha &= \{x \in X^\alpha \colon (\exists m)(\forall n >m)(\exists j)(\forall k >j)(
x(n)(k) \in \FF^{\beta +n-2})\},\\
\II^\alpha &= \{x \in X^\alpha \colon (\exists m)(\forall n >m)(\exists j)(\forall k >j)(
x(n)(k) \in \II^{\beta +n-2})\}.
\end{align*}
By induction we can assume that we have disjoint $\SB{\beta +n-1}$ sets $A$
and $B$ such that $\FF^{\beta +n-2} \subset A$ and $\II^{\beta +n-2} \subset B$. Hence 
\begin{align*}
\FF^\alpha &\subset \{x \in X^\alpha \colon (\exists m)(\forall n >m)(\exists j)(\forall k >j)(
x(n)(k) \in A)\},\\
\II^\alpha &\subset \{x \in X^\alpha \colon (\exists m)(\forall n >m)(\exists j)(\forall k >j)(
x(n)(k) \in B)\}.
\end{align*}
Thus again by Lemma~\ref{lem: separating core} we have that $\FF^\alpha$ and $\II^\alpha$ can be separated by disjoint $\SB{\alpha+1}$ sets. 
\end{proof}

The next goal is to show that if $\alpha \ge 1$, the separation obtained in the previous lemma is optimal. The case for $\alpha=1$ is simple.
\begin{lemma}
\label{lem: 1 not separable}
There do not exist disjoint $\PB{2}$ sets $A$ and $B$ such that $\FF^1 \subset A$ and $\II^1 \subset B$.
\end{lemma}
\begin{proof}
If $\FF^1 \subset A$ then $A$ must be comeager (as any dense $\PB{2}$ set is comeager). If $\II^1 \subset B$ then $B$ is comeager. Hence $A \cap B \ne \emptyset$.
\end{proof}

For $\alpha >1$, our plan is to take $x \in X^1$ and continuously encode
$x$ into an element $\rho(x)$ of $X^\alpha$ such that if $x \in \FF^1$, then
$\rho(x) \in \FF^\alpha$ and if 
$x \in \II^1$, then $\rho(x) \in \II^\alpha$.
We want to do this in such a way that,  relative to some parameter $p$,
$x$ can uniformly compute the $\alpha$\textsuperscript{th} iterate of the Turing jump of
$\rho(x)$ (or the $(\alpha -1)$\textsuperscript{th} iterate of the jump if $\alpha <\omega$). From this it will follow that 
if $\FF^\alpha$ and $\II^\alpha$ are separable by $\DL{\alpha +1}$ sets, then 
$\FF^1$ and $\II^1$ are separable by $\DL{2}(p)$ sets which we know by Lemma~\ref{lem: 1 not separable} is impossible.

We will introduce some of  the main ideas needed with an example. 
Let $T$ be the full $\omega$-branching tree of height $2$.  Hence the nodes of $T$ are all strings of natural numbers of length 
$0$, $1$, or $2$. We will denote the root of the tree by $\emptystring$.
We are now going to label all nodes of $T$ except $\emptystring$ with
either $0$ or $1$. This is a function $f:T\setminus \{\emptystring\} \rightarrow 2$. We will encode $x$ at the first level of the tree so for all  $i \in \N$, $f(i) = x(i)$. 
At the second level of the tree, we will label nodes so that for all $i$, 
$\lim_j f(i \concat j)$ exists and is equal to $f(i)$. 
In order to complete our definition of $f$ we need some additional information as to what values $f(i\concat j)$ should take before the limit is reached. For this, we will take an additional function 
$g$ that maps from the nodes of $T$, that are neither the root nor a leaf, to $2^{<\omega}$. For our example tree,  $g$ is a function from nodes of length $1$ to $2^{<\omega}$.  Now we can complete our definition of $f$ as follows. Fix $i$.
\begin{enumerate}
 \item If $j < |g(i)|$, the $f(i \concat j) = g(i)(j)$.
\item If $j = |g(i)|$, the $f(i \concat j) = 1 - f(i)$.
 \item If $j > |g(i)|$, the $f(i \concat j) = f(i)$.
 \end{enumerate}
Defining $f(i \concat j) = 1 - f(i)$ for $j = |g(i)|$ allows us to recover $g(i)$ from the function $j \mapsto f(i \concat j)$.

Now consider $f$, restricted to the leaf nodes. We can regard $f$ as an element of $X^2$ by  mapping $f$ to $z \in X^2$ where $z(i)(j) = f( i \concat j)$. Further, if $x \in \FF^1$, then $z \in \FF^2$, and if $x \in \II^1$, then $z \in \II^2$. Hence if we fix $g$, we have a mapping 
$\rho: X^1 \rightarrow X^2$ with the desired properties. 
We will show that if $g$ is sufficiently generic, then \textit{for any $x$}, $\rho(x)'$ is uniformly Turing reducible to $x \oplus g \oplus 0'$. Given this, assume that  $\FF^2$ and $\II^2$ are separable by disjoint $\DL{3}$ sets $A$ and $B$ with 
$\FF^2 \subset A$ and $\II^2 \subset B$. 
This means that the double jump of $\rho(x)$ can uniformly determine whether $\rho(x)$ is in $A$ or $B$, and hence  $(g\oplus x \oplus 0')'$ can uniformly determine if $\rho(x)$ is in $A$ or $B$. As $g$ is independent of $x$, this means that the sets  $\{x \in X^1 \colon \rho(x) \in A\}$ and 
$\{x \in X^1 \colon \rho(x) \in B\}$ are $\DL{2}(g \oplus 0')$. But these sets separate $\FF^1$ from $\II^1$, a contradiction.  By relativization, we get that $\FF^2$ and $\II^2$ are not separable by  $\DB{3}$ sets. As separation by disjoint  $\PB{\alpha}$ sets implies separation by disjoint  $\DB{\alpha}$ sets, we conclude that 
$\FF^2$ and $\II^2$ are not separable by  $\PB{3}$ sets.

Our proof of Theorem~\ref{thm: non-separation} is written in the language of effective descriptive
set theory. We will establish a lightface version of this theorem, and then
obtain the full result by relativization. Let us start by defining a
wellfounded tree $T_\alpha \subset \omega^{< \omega}$ for each computable
ordinal $\alpha$. Let $T_0$ be the tree consisting of just the empty
string, so $T_0 = \{\emptystring\}$. For $\alpha > 0$, let $T_\alpha$ be the
tree having  $\omega$ many branches at the root with the tree above the
$n$\textsuperscript{th} node being $T_{\pi_\alpha(n)}$, so $n \concat \sigma \in T_{\alpha}$
if $\sigma \in T_{\pi_\alpha(n)}$. We can assume that the functions
$\pi_\beta$, for $\beta \le \alpha$, are all uniformly computable.  Note
that for each $\alpha$, the tree $T_\alpha$ has rank $\alpha$.

In what follows, fix a computable ordinal $\gamma$ and take $T$ to be
$T_\gamma$. For all $\alpha \le \gamma$ we define the following subsets of
$T$.
\begin{align*}
N_\alpha &= \{ \sigma \in T \colon \rank(\sigma) = \alpha\}\\
L_\alpha &= \{ \sigma \in T \colon \rank(\sigma) < \alpha\}\\
A_\alpha &= \{ \sigma \in T \colon \rank(\sigma) < \alpha \land
\rank(\sigma^{-}) \ge \alpha\}
\end{align*}
Here by $\sigma^{-}$ we mean $\sigma$ without its last bit. 

Note that $N_0=A_1$ and is the set of all leaf nodes of  $T$. However, if 
$\gamma \ge \omega$, then $N_1 \subsetneq A_2$. 

The following lemma follows immediately from these definitions.
 
\begin{lemma}\
\begin{enumerate}
 \item $N_\alpha \subset A_{\alpha+1}$.
 \item $A_{\alpha+1}\setminus N_{\alpha} \subset A_\alpha$. 
 \item For each $0 < \alpha \le \gamma$ the set $A_\alpha$ is a maximal anti-chain.
 \item $A_\gamma$ is the set of successors of the root of $T_\gamma$.  \qed
\end{enumerate}
\end{lemma}

Because $T_\gamma$ is constructed using the same functions $\pi_\alpha$, for $\alpha \le \gamma$ as $X^\gamma$, there is a natural homeomorphism from $2^{N_0}$ to $X^\gamma$. This is defined as follows. We map $f \in 2^{N_0}$ to $x \in X^\gamma$ where for all sequences $i_1\concat i_2\concat \ldots \concat i_k \in N_0$ we have that 
\[x(i_1)(i_2)\ldots(i_k) = f(i_1\concat i_2\concat \ldots \concat i_k).\]
To simplify the exposition, let us identify $X^\gamma$ with $2^{N_0}$ and so we can regard both $\FF^\gamma$ and $\II^\gamma$ as subsets of $2^{N_0}$. 

We will now define a family of continuous maps $\rho_g$ from $X^1$ to
$X^\gamma$. 
Suppose $x\in X^1$ and $g$ is a function from $T  \setminus (N_0 \cup
\{\emptystring\})$ to $2^{<\omega}$. Then there is a unique  mapping
$f_{g,x}: T \setminus \{\emptystring\} \rightarrow 2$  with the following properties. 
\begin{enumerate}
 \item For all $i \in \N$, $f_{g,x}(i) = x(i)$. 
 \item For all nodes $\sigma \in T \setminus (N_0 \cup \{\emptystring\})$,  if $|g(\sigma)|=n$, then 
 \begin{enumerate}
 \item  For all $i < n$, $f_{g,x}(\sigma\concat i) = g(\sigma)(i)$
 \item $f_{g,x}(\sigma \concat n) = 1-f_{g,x}(\sigma)$
 \item For all $i >n$, $f_{g,x}(\sigma\concat i) = f_{g,x}(\sigma)$. 
\end{enumerate}
\end{enumerate}
Our idea is for that each $\sigma$, the values of $f_{g,x}(\sigma \concat
i)$ for $i \in \omega$ should eventually reach a limit equal to
$f_{g,x}(\sigma)$. The finitely many values before the limit is reached are
specified by $g(\sigma)$ 

If we fix $g$, we now have a continuous mapping $\rho_g:X^1 \rightarrow
X^\gamma$ defined by $\rho_g(x) = f_{g,x}\until{N_0}$. This mapping is not surjective, but it does have the following important property.
\begin{lemma}
If $x \in \FF^1$, then $\rho_g(x) \in \FF^\gamma$ and if $x \in \II^1$,
then $\rho_g(x) \in \II^\gamma$. \qed
 \end{lemma}

We will soon examine what happens to the mapping $\rho_g$ if we take $g$ to
be sufficiently generic. In the following, by generic, we will always mean for the
Cohen partial order of finite partial functions ordered by inclusion. In
particular, suppose $f$ is a function between computable sets $A$ and $B$
(for example, from $T \setminus (N_0 \cup \{\emptystring\})$ to $2^{<\omega}$),
and $C$ is a computable subset of $A$.  Then we say that $f \until{C}$ is
\define{arithmetically generic relative to $z \in \cantor$}, if for every
arithmetically definable dense open set $D$ of finite partial functions
from $C$ to $B$, there is a finite partial function $p \subset f \until{C}$
such that  $p$ meets $D$. We say that $f,g$ are mutually $z$-generic,
if $f \oplus g$ is $z$-generic, where $f \oplus g$ is defined on the disjoint union of the domains of $f$ and $g$.

Suppose $h$ is a function from $\omega$ to $2^{< \omega}$, and $y \in
\cantor$. Fix a computable pairing function $\seq{\cdot,\cdot} \from
\omega^2 \to \omega$.
Define $h_{y} \in \cantor$ as follows. For all $i,j$, 
\[h_{y}(\seq{i,j}) = 
\begin{cases}
 h(i)(j) &\mbox{if }j <| h(i)|\\
 1- y(i) &\mbox{if }j = | h(i)|\\
 y(i)& \mbox{otherwise.}
\end{cases}\]
Once again, here we are coding the real $y$ into $h_y$ where each bit
$y(i)$ is the limit $\lim_{j \to \infty} h_y(\seq{i,j})$, and the finitely
many values before this limit is reached are specified by $h(i)$. We now
have the following lemma from \cite[Lemma 2.3]{MSS}, where $x$ from that
paper corresponds to our $h$, $y$ from that paper corresponds to our $y$,
$J(x,y)$ is $h_y$, and we state our lemma only for the case $n = 1$ and
relativized to a real $z$. 

Recall finally that if $x \in \cantor$, we use $x'$ to denote the Turing
jump of $x$, and if $\alpha < \ck$, we use $x^{(\alpha)}$ to
denote the $\alpha$\textsuperscript{th} iterate of the Turing jump relative to $x$. 

\begin{lemma} 
\label{lem: next level}
Suppose $y,z \in \cantor$, and $g,h \from \omega \to 2^{< \omega}$ are mutually
arithmetically generic relative to $z$. 
Then $(g \oplus h_y \oplus z)' \equiv_T g \oplus h \oplus y \oplus z'$
uniformly.
 \end{lemma}

We emphasize that for the previous lemma to hold, $g$ and $h$ do not need to be generic relative to $x$. The proof of the next lemma is standard.

\begin{lemma}
\label{lem: generic subset}
  Suppose $x,z \in \cantor$, and $g \from T  \setminus (N_0 \cup
\{\emptystring\}) \to 2^{<\omega}$ is arithmetically generic relative to $z$. 
 Let $E$ be an infinite computable subset of $T$ such that 
 \begin{enumerate}
 \item $E \cap (A_\gamma \cup \{\emptystring\}) = \emptyset$. 
 \item No infinite subset of $E$ shares a common predecessor
\end{enumerate}
 Let $A$ be an infinite computable subset of $T \setminus (N_0 \cup \{\emptystring\})$ such that no element of $A$ is a predecessor of an element of $E$. 
Then $g\until{A}$ and $f_{g,x}\until{E}$ are mutually arithmetically $z$-generic. 
\end{lemma}
\begin{proof}
Suppose $D$ is an dense open set in the partial order for building $g \until{A} \oplus f_{g,x}
\until{E}$ which is arithmetically definable from $z$. Let $\P$ be the partial order for building $g$: the order of
finite partial functions from $(N_0 \cup
\{\emptystring\})$ to $2^{<\omega}$. We will define
a dense open set $D^*$ in $\P$ so
that if $g$ meets $D^*$, then $g \until{A} \oplus f_{g,x} \until{E}$ meets~$D$. 

Suppose $p \in \P$. 
Let $P$ be the set of elements of $E$ that have a predecessor in $\dom(p)$.
Note that if $\sigma \in P$, then we cannot effect the value of $f_{g,x}(\sigma)$ by extending
$p$. However, $P$ is finite and for any string $\sigma \in E \setminus P$ with $\sigma = n
\concat \tau$, if we extend $p$
to some $p^*$ where $p^*(\tau)$ has length $\geq n$, then $f_{g,x}(\sigma)
= \tau(n)$. By considering every possible value of $f_{g,x} \until{P}$ and
iteratively finding extensions for each of them,
we can find an
extension $p^*$ of $p$ so that $g \until{A} \oplus f_{g,x} \until{E}$ meets
$D$, no matter what the value of $f_{g,x} \restriction P$ is. Now define
$D^*$ to be the set union over all $p \in \P$ of such strings~$p^*$. 
\end{proof}

It does not matter for the purposes of our proof, but a more precise
calculation here is that if $g$ is $2$-generic relative to $z$, then $g
\until{A}$ and $f_{g,x} \until{E}$ will be mutually $1$-generic relative to
$z$. Similarly, in Lemma~\ref{lem: next level}, $g$ and $h$ are only
required to be $2$-generic for the conclusion to hold.

Following is our main technical lemma. 

\begin{lemma}
\label{lem: main}
  Suppose $x, z \in \cantor$, $\alpha < \gamma$, $g \from T  \setminus (N_0 \cup
\{\emptystring\}) \to 2^{<\omega}$ is arithmetically generic relative to $z$. 
 Then uniformly 
 \[(f_{g,x}\until{A_\alpha} \oplus g\until{L_\alpha} \oplus z)' \le_T
 f_{g,x}\until{A_{\alpha+1}} \oplus g\until{L_{\alpha+1}} \oplus z'.\]
\end{lemma}
\begin{proof}
Separate $A_{\alpha}$ into two sets $D$ and $E$. Let $D$ be those elements of $A_\alpha$ who have predecessors in  $N_{\alpha}$ and let $E$ be the rest. Note that if $\sigma\concat i \in E$, then $\rank(\sigma) > \alpha \ge \rank(\sigma\concat i)+1$ and hence $\rank(\sigma)$ is a limit ordinal. Thus for some $m$, for all $n$ greater than $m$, $\rank(\sigma \concat m) >\alpha$ (as the tree is defined to have strictly increasing ranks at limits).   
It follows that for any $\sigma \in T$, there are not infinitely many
successors of $\sigma$ in $E$. By considering ranks, it follows that no element of $L_{\alpha+1}$ is a predecessor of an element of $E$. Hence by 
Lemma~\ref{lem: generic subset}, $g\until{L_{\alpha+1}}$ and  $f_{g,x}\until{E}$ are mutually arithmetically $z$-generic. As $L_{\alpha+1}$ is the disjoint union of $L_\alpha$ and $N_\alpha$, it follows that $g\until{N_\alpha}$ and 
$g\until{L_\alpha} \oplus f_{g,x}\until{E}$ are mutually arithmetically $z$-generic.

If we let $h = g\until{N_{\alpha}}$ and $y =
 f_{g,x}\until{N_{\alpha}}$, then $h_y = f_{g,x}\until{D} $ (under some uniform computable bijection) using the notation of  Lemma~\ref{lem: next level}. 
 Hence applying Lemma~\ref{lem: next level}, we have that,
\begin{align*}
(f_{g,x}\until{A_\alpha} \oplus g\until{L_\alpha} \oplus z)' 
&\equiv_T (  (f_{g,x}\until{E} \oplus g\until{L_\alpha})\oplus f_{g,x}\until{D} \oplus z)'\\
&\le_T f_{g,x}\until{E} \oplus g\until{L_\alpha} \oplus g\until{N_{\alpha}} \oplus  f_{g,x}\until{N_{\alpha}} \oplus z'.
\end{align*}
 Now as $L_{\alpha +1} = L_\alpha \cup N_{\alpha}$, and $A_{\alpha+1} = N_{\alpha} \cup E$. We obtain the desired result. 
\end{proof}

\begin{lemma}
  Suppose $x, z \in \cantor$, and $g \from T  \setminus (N_0 \cup
\{\emptystring\}) \to 2^{<\omega}$ is arithmetically generic relative to $z$. 
If $1 \le n < \omega$ and $n \leq \gamma$, then uniformly 
\[f_{g,x}\until {A_{n+1}}
\oplus g\until{L_{n+1}} \oplus0^{(n)} \ge_T (f_{g,x}\until{N_{0}})^{(n)}.\]
\end{lemma}
\begin{proof}
The result holds  if $n=0$ because $A_1=N_0$.  Consider the case $n+1$. We have,  by the previous lemma, and the  induction hypothesis, that uniformly
\begin{align*}
 f_{g,x}\until{A_{n+2}}\oplus g\until{L_{n+2}}  \oplus0^{(n+1)} 
 & \ge_T (f_{g,x}\until{A_{n+1}}\oplus g\until{L_{n+1}} \oplus 0^{(n)})'\\
 & \ge_T (f_{g,x}\until{N_0})^{(n+1)}. \qedhere
\end{align*}
\end{proof}

\begin{lemma}
  Suppose $x, z \in \cantor$, and $g \from T  \setminus (N_0 \cup
\{\emptystring\}) \to 2^{<\omega}$ is arithmetically generic relative to
$0^{(\alpha)}$.
If $\omega \le  \alpha \le \gamma$, then uniformly
\[f_{g,x}\until {A_\alpha} \oplus g\until{L_\alpha} \oplus 0^{(\alpha)} \ge_T
(f_{g,x}\until{N_{0}})^{(\alpha)}.\]
\end{lemma}
\begin{proof}
Note that if $\alpha>\beta$ then uniformly in $\alpha$ and $\beta$ we have
that $f_{g,x}\until{A_\beta}$ is computable from $f_{g,x}\until{A_\alpha} \oplus g\until{L_\alpha}$. 
Consider the case that $\alpha =\omega$. By the previous remark,  $f_{g,x}\until {A_\omega} \oplus g\until{L_\omega} \oplus
0^{(\omega)}$ uniformly computes $f_{g,x}\until {A_n} \oplus g\until{L_n}
\oplus 0^{(n)}$ for all $n$. Hence by Lemma~\ref{lem: main}, $f_{g,x}\until {A_\omega} \oplus g\until{L_\omega} \oplus
0^{(\omega)}$ uniformly computes
$(f_{g,x}\until{N_{0}})^{(n)}$ and so $(f_{g,x}\until{N_{0}})^{(\omega)}$.
For successors, just last repeat the previous lemma, and at limits repeat the argument for $\omega$. 
\end{proof}

The difference between the previous two lemmas  corresponds to a slight difference between the indexing of light-face Borel sets and  the iterates of the Turing jump. This discrepancy also occurs in the following standard lemma.

\begin{lemma}
A set $X \subset 2^\omega$ is $\DL{\alpha}$ if and only if there is an $e$ such that 
\begin{enumerate}
 \item For all $x$, $\Phi_e(x^{(\beta)} ; 0) \downarrow$.
 \item For all $x$, $\Phi_e(x^{(\beta)} ; 0) =1$ if and only if $x \in X$.
\end{enumerate}
 Where $\beta = \alpha$ if $\alpha \ge \omega$ and $\beta = \alpha -1$ if $ 1 \le\alpha <\omega$. 
\end{lemma}

\begin{lemma}If $q \leq \gamma \leq \ck$, 
then the sets $\FF^\gamma$ and $\II^\gamma$ are not separable by $\DL{\gamma+1}$ sets.
\end{lemma}
\begin{proof}
 Assume that $A$ and $B$  are disjoint $\DL{\gamma+1}$ sets  with 
 $\FF^\gamma \subset A$ and $\II^\gamma \subset B$. We can additionally
 assume that $A \cup B = X^\gamma$.   Fix $g$ which is arithmetically
 generic relative to $0^{(\gamma)}$. We have the associated function
 $\rho_g\from X^1\to X^\gamma$ as above. 
 
First consider if $\gamma \ge \omega$. Then from $(\rho_g(x))^{(\gamma+1)}$
we can compute whether $\rho_g(x) \in A$ or $\rho_g(x) \in B$.  Hence we can compute this from  
 $(x \oplus g\until{L_\gamma} \oplus 0^{(\gamma)})'$. Note that
 $g\until{L_\gamma} \oplus 0^{(\gamma)}$ is fixed for all $x$. Hence $\FF^1$ and $\II^1$ are separable by 
 $\DL{2}(g\until{L_\gamma} \oplus 0^{(\gamma)})$ sets. But this is a contradiction.
If $1 \le \gamma  <\omega$ then the same argument works after replacing occurrences of $\gamma$ with $\gamma -1$.
\end{proof}

Theorem~\ref{thm: non-separation} follows immediately from relativizing this lemma and observing that if $\FF^\gamma$ and $\II^\gamma$ could be separated by $\PB{\gamma+1}$ sets, then they could be separated by $\DB{\gamma+1}$ sets. 
 
\bibliographystyle{plain}
\bibliography{references}

\end{document}